\documentclass[11pt]{amsart}

\usepackage{graphicx,hyperref}
\usepackage{tikz}
\usepackage{tikz-cd}

\usepackage[titletoc]{appendix}

\theoremstyle{plain}
  \newtheorem{theorem}{Theorem}
  
  \newtheorem{proposition}{Proposition}
  \newtheorem{corollary}[proposition]{Corollary}
  \newtheorem{lemma}[proposition]{Lemma}
\theoremstyle{definition}
  \newtheorem{ex}{Example}
  \newtheorem{rem}{Remark}

\newcommand{\Dh}{\mathrm{DH}}
\renewcommand{\H}{\mathcal{H}}
\newcommand{\R}{\mathbb R}
\newcommand{\RR}{{\mathbb R}}
\newcommand{\C}{\mathbb C}
\newcommand{\Cone}{\mathcal{C}}

\newcommand{\glnc}{\mathrm{GL}(n,\C)}
\newcommand{\Un}{\mathrm{U}(n)}
\newcommand{\B}{\mathcal{B}}

\newcommand{\T}{\mathbb T}

\newcommand{\captionfonts}{\small}

\newcommand{\conez}{\Delta_0}
\newcommand{\coneb}{\Cone_{\B}}
\newcommand{\ctrop}{\Cone_{\T}^\circ}
\newcommand{\conekt}{\Cone_{\mathrm{KT}}}
\newcommand{\conegz}{\Delta{}_{\mathrm{GZ}}}
\newcommand{\rntri}{\R^{3n}}

\newcommand{\coneho}{\Cone_{\mathrm{Horn}}}

\newcommand{\lbb}{l^\B}
\newcommand{\ls}{\sigma}

\renewcommand{\AA}{Q}
\newcommand{\uu}{\mathfrak{u}}
\newcommand{\LL}{\Lambda}
\newcommand{\WW}{\widehat{W}}
\newcommand{\ctropfull}{{\Cone}_\T}
\makeatletter  
\long\def\@makecaption#1#2{%
  \vskip\abovecaptionskip
  \sbox\@tempboxa{{\captionfonts #1: #2}}%
  \ifdim \wd\@tempboxa >\hsize
    {\captionfonts #1: #2\par}
  \else
    \hbox to\hsize{\hfil\box\@tempboxa\hfil}%
  \fi
  \vskip\belowcaptionskip}
\makeatother   

\begin{document}
\author{Anton Alekseev}

\address{Anton Alekseev, Department of Mathematics, University of Geneva, 2-4 rue du Li\`evre,
c.p. 64, 1211 Gen\`eve 4, Switzerland}

\email{Anton.Alekseev@unige.ch}

\author{Masha Podkopaeva}

\address{Maria Podkopaeva, Department of Mathematics, University of Geneva, 2-4 rue du Li\`evre,
c.p. 64, 1211 Gen\`eve 4, Switzerland \& St. Petersburg State University, Department of Mathematics and Mechanics, Chebyshev Laboratory, 14 Liniya VO 29B, 199178 St.~Petersburg, Russia}
\email{Maria.Podkopaeva@unige.ch}

\author{Andras Szenes}

\address{Andras Szenes, Department of Mathematics, University of Geneva, 2-4 rue du Li\`evre,
c.p. 64, 1211 Gen\`eve 4, Switzerland}

\email{Andras.Szenes@unige.ch}
\title[]{A symplectic proof of the Horn inequalities}

\begin{abstract}
In this paper, we give a symplectic proof of the Horn inequalities on eigenvalues of a sum
of two Hermitian matrices with given spectra. Our method is a combination of tropical calculus
for matrix eigenvalues, combinatorics of planar networks, and estimates for the Liouville volume.
As a corollary, we give a tropical description of the Duistermaat--Heckman measure on the Horn
polytope.

\end{abstract}
\maketitle
\section{Introduction}

\subsection{The Horn problem}\label{hornintro}
Fix a positive integer $n$, and let $\H$ be the set of Hermitian
matrices of size $n$. For $K \in \H$, denote by $\lambda(K)=$ \mbox{$(\lambda_1
\geq \lambda_2, \geq \dots \geq \lambda_n)$} the set of eigenvalues of
$K$ listed in decreasing order, and introduce the map $l: \H\to
\R^{n}$ defined by the equalities
\[
l_1(K) = \lambda_1,\  l_2(K)= \lambda_1 + \lambda_2,\ \dots,\ 
l_n(K)= \lambda_1 + \dots + \lambda_n = {\rm Tr}(K).
\]

 We will call the set
\begin{equation}\label{horncone}\begin{split}
  \Cone_{\mathrm{Horn}}& = \{(a,b,c)\in\rntri;\ 
\exists(K_1,K_2)\in\H^{\times2}: \\ & l(K_1)=a,
  l(K_2)=b,l(K_1+K_2)=c\}\end{split}
\end{equation}
the {\em Horn cone}. 

Clearly, $\Cone_{\mathrm{Horn}}$ is a closed subset of the hyperplane
\[
\{(a,b,c)\in\rntri;
a_n+b_n=c_n\}\subset\rntri,
\] and $\tau
\Cone_{\mathrm{Horn}}=\Cone_{\mathrm{Horn}}$ for any $\tau>0$.

The problem of determining this cone, known as the Horn problem, has a
long history (see \cite{fulton} for details).   The first conjectural description 
was given by Horn \cite{horn} in 1962; it presents
$\Cone_{\mathrm{Horn}}$ as the set of solutions of a complicated, recursively defined list of
linear inequalities. This description, in particular, implies that
$\Cone_{\mathrm{Horn}}$ is a closed convex cone. Later,  a natural
explanation for this fact was found in terms of convexity properties of moment maps in symplectic
geometry.   

In 1999, Knutson and Tao came up with the following much simpler,
albeit implicit description of $\Cone_{\mathrm{Horn}}$ (see \cite{KT},
\cite{Buch}).  Consider the regular triangulation of order $n$ of an
equilateral triangle. The triangle is divided into $n^2$ small
triangles. Two adjacent triangles form a rhombus, which can be of one
of the three types shown in Figure~\ref{hive}.

\begin{figure}[h]\label{hive}\begin{center}
\begin{tikzpicture}
\fill[black] (0,6) circle (2pt) node[above]{$l_0^n$};
\fill[black] (1,6) circle (2pt) node[above]{$l_1^n$};
\fill[black] (2,6) circle (2pt);
\fill[black] (3,6) circle (2pt);
\fill[black] (4,6) circle (2pt);
\fill[black] (5,6) circle (2pt) node[above]{$l_n^n$};
\fill[black] (0.5,5.3) circle (2pt);
\fill[black] (1.5,5.3) circle (2pt);
\fill[black] (2.5,5.3) circle (2pt);
\fill[black] (3.5,5.3) circle (2pt);
\fill[black] (4.5,5.3) circle (2pt);
\fill[black] (1,4.6) circle (2pt);
\fill[black] (2,4.6) circle (2pt);
\fill[black] (3,4.6) circle (2pt);
\fill[black] (4,4.6) circle (2pt);
\fill[black] (1.5,3.9) circle (2pt);
\fill[black] (2.5,3.9) circle (2pt);
\fill[black] (3.5,3.9) circle (2pt);
\fill[black] (2,3.2) circle (2pt) node[left]{$l_0^1$};
\fill[black] (3,3.2) circle (2pt) node[right]{$l_1^1$};
\fill[black] (2.5,2.5) circle (2pt) node[below]{$l_0^0$};
\draw[opacity=0.3] (0,6) -- (5,6);
\draw[opacity=0.3] (0.5,5.3)  -- (4.5,5.3);
\draw[opacity=0.3] (1,4.6) -- (4,4.6);
\draw[opacity=0.3] (1.5,3.9) -- (3.5,3.9);
\draw[opacity=0.3] (2,3.2) -- (3,3.2);
\draw[opacity=0.3] (0,6) -- (2.5,2.5) -- (5,6);
\draw[opacity=0.3] (1,6) -- (3,3.2);
\draw[opacity=0.3] (2,6) -- (3.5,3.9);
\draw[opacity=0.3] (3,6) -- (4,4.6);
\draw[opacity=0.3] (4,6) -- (4.5,5.3);
\draw[opacity=0.3] (4,6) -- (2,3.2);
\draw[opacity=0.3] (3,6) -- (1.5,3.9);
\draw[opacity=0.3] (2,6) -- (1,4.6);
\draw[opacity=0.3] (1,6) -- (0.5,5.3);
\draw[ultra thick] (1,4.6) -- (2,4.6) -- (2.5,5.3) -- (1.5,5.3) -- cycle;
\draw[ultra thick, dotted] (3.5,5.3) -- (4.5,5.3) -- (4,6) -- (3,6) -- cycle;
\draw[ultra thick, dashed] (3,3.2) -- (3.5,3.9) -- (3,4.6) -- (2.5,3.9) -- cycle;
\end{tikzpicture}\end{center}
\caption{The triangular tableau with three types of rhombi.}\label{hive}
\end{figure}
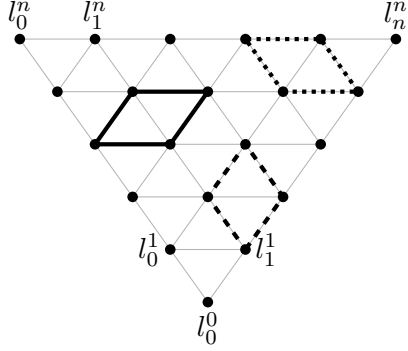

We will call the assignment of a real number to each of the nodes of
the triangulation a {\em tableau}. Denoting by $\nabla$ the set of
nodes of the triangulation, we can identify the space of tableaux with $\R^\nabla$.

Let $l_i^k$ be the number at the $i$th node in the $k$th row of the
triangulation, $0 \leq i \leq k \leq n$. Then each rhombus gives rise
to an inequality: the sum of the two numbers assigned to the endpoints
of the short diagonal is greater than or equal to the sum of the two
numbers assigned to the endpoints of the long diagonal. A tableau is
called a {\em hive} if it satisfies all the inequalities, i.e., if for
$0 < i \leq k < n$,
\begin{equation}  \label{hiveineq}
\begin{array}{lll}
l^{k+1}_{i} + l^{k}_{i-1} & \geq & l^{k+1}_{i-1} + l^{k}_{i},  \\
l^{k+1}_{i} + l^{k}_{i} & \geq & l^{k+1}_{i+1} + l^{k}_{i-1},\\
l^{k}_{i} + l^{k}_{i-1} & \geq & l^{k+1}_{i} + l^{k-1}_{i-1}.
\end{array}
\end{equation}

Clearly, the set of hives $\Cone_3$, defined by the three sets of
inequalities \eqref{hiveineq}, is a closed cone in $\R^\nabla$.
Now consider the boundary map: $\partial:\R^\nabla\to\R^{3n}$ given by
\begin{equation*}
  \label{eq:1}
   a_i=l^n_i ,\ b_i=l^{n-i}_{n-i}-l^n_n,\
       c_i = l^{n-i}_0,\ 1\leq i\leq n
\end{equation*}
Denote by $\conekt = \partial(\Cone_3)$ the polyhedral cone obtained
as the image of $\Cone_3$ along this map.

\begin{theorem}[Knutson--Tao] 
\label{knutao}
The Horn cone coincides with the Knutson--Tao cone: $\coneho=\conekt$.
\end{theorem}

Speyer \cite{speyer} gave another proof of this theorem using Viro's
patchworking and Vinnikov curves. The purpose of the present article
is to provide a proof based on a combination of ideas from tropical and
symplectic geometry.

\subsection{The multiplicative problem}

There is a similar multiplicative problem defined for the group $\B$
of complex upper-triangular matrices of size $n$ with positive entries
on the diagonal.

For a matrix $A\in \B$, the {\em singular values} are defined as the
eigenvalues $\lambda_i(AA^*)$, $i=1,\dots,n$, of the matrix $AA^*$,
which are positive real numbers in this case. The map 
 \begin{equation}
   \label{notationl}
\lbb:\B\to\R^{n},\quad \lbb_i(A)=\frac12 \sum_{k=1}^i \log \,
\lambda_k(AA^*),\,i=1, \dots, n 
 \end{equation}
is intertwined with the map $l:\H\to\R^n$ by the diffeomorphism between $\H$ and $\B$
given by $\exp(2K) = AA^*$.

We can also define the multiplicative analog of the Horn cone:
\begin{equation}\label{coneb}\begin{split}
\coneb = \{(r,s,t)\in\rntri;\ & \exists(A,C)\in \B^{\times2};\\ &\lbb(A)=r, \lbb(C)=s, \lbb(AC)=t\}.
\end{split}\end{equation}
The following surprizing result (also known as the Thompson Conjecture) was proved by 
Klyachko \cite{Klyachko}:

\begin{theorem}[Klyachko]\label{klyachko}
The set $\coneb$ coincides with the Horn cone: $$\coneb = \coneho.$$
\end{theorem}

In particular, this implies that $\coneb$ is a polyhedral cone.

\subsection{Planar networks}\label{planar}
We define one more subset of $\rntri$, this time using the theory of planar networks.

Recall that a  planar network is the following data:
  \begin{itemize}
  \item a finite  oriented planar graph $\Gamma$ with vertex set $V\Gamma$ and edge
    set $E\Gamma$,
  \item an embedding of $\Gamma$ into the strip $\{x_0\leq x \leq
    x_1\}\subset\R^2$ such that the image of each edge is a segment of
    a straight line, which is not parallel to the $y$-axis. This
    condition allows us to define an orientation of $\Gamma$: we
    orient each edge in the positive direction along the $x$-axis.
  \end{itemize}
  The vertices on the line $\{x=x_0\}$ are called {\em sources} and
  the vertices on the line $\{x=x_1\}$ are called {\em sinks} of
  $\Gamma$. A planar network with $n$ sources and $n$ sinks is called
  a planar network {\em of rank $n$}. Without loss of generality, we
  can assume that the set of $y$-coordinates of the sources and sinks
  is the set of the first $n$ integers $\{1,2,\dots,n\}$.

  A $k$-path in $\Gamma$ is a collection of $k$ vertex-disjoint
  oriented paths connecting $k$ sources with $k$ sinks. The set of
  $k$-paths in $\Gamma$ is denoted by $P_k\Gamma$. For $I, J \subset
  \{ 1, \dots, n\}$, two subsets of cardinality $k$, we denote by
  $P_k\Gamma(I, J)$ the set of $k$-paths with the sources correponding
  to $I$ and sinks corresponding to $J$.

  Let $\AA$ be an abelian semigroup with unit, and let $W(\Gamma, \AA)$ be
  the set of weightings of $\Gamma$ with values in this semigroup: $W(\Gamma,
  \AA)=\AA^{E \Gamma}$. For \mbox{$w \in W(\Gamma, \AA)$}
  and a collection of edges $\alpha \subset E\Gamma$, we set
$$
w(\alpha) = \prod_{e \in \alpha} w(e).
$$
If $\alpha = \emptyset$, then we set $w(\alpha)=1_\AA$.

\begin{ex}\label{phases}
  When $\AA=U(1)$ is the group of unitary complex numbers, then we will write
  $\Phi(\Gamma) = W(\Gamma, U(1))$. For a weighting $\phi \in
  \Phi(\Gamma)$, we have
$$
\phi(\alpha) = \prod_{e \in \alpha} \phi(e).
$$
\end{ex}

\begin{ex}
  Consider the tropical semigroup $\T=\R \cup \{ - \infty\}$ with
  group law given by addition: $(x,y) \mapsto x+y $. Then for $w \in
  W(\Gamma, \T)$ we have
\begin{equation}
    \label{tropdef}
w(\alpha) = \sum_{e \in \alpha} w(e).
\end{equation}
\end{ex}

\subsection{Correspondence map}

Let $\Gamma$ be a  planar network of rank $n$ and $w\in W(\Gamma, \AA)$ 
a weighting of $\Gamma$ with values in a commutative semiring
$\AA$. To this pair, we can associate an $n$-by-$n$ matrix with matrix
elements in $\AA$:
\begin{equation}
  \label{correspondence}
  M_{i,j}(\Gamma; w)= \sum_{\alpha \in P_1\Gamma(i,j)} w(\alpha) =
\sum_{\alpha \in P_1(i,j)} \prod_{e\in \alpha} w(e).
\end{equation}

In case the set of paths $P_1\Gamma(i,j)$ is empty, we set $M_{i,j}(\Gamma; w)$
equal to the additive unit (zero) of $\AA$.

Let $\Gamma_1$ and $\Gamma_2$ be two rank-$n$ planar networks and let
$\Gamma=\Gamma_1 \circ \Gamma_2$ be their concatenation, i.e.
$\Gamma$ is a union of $\Gamma_1$ and $\Gamma_2$ with sinks of
$\Gamma_1$ identified with sources of $\Gamma_2$. Then a pair of
weightings $w_1 \in W(\Gamma_1, \AA),\ w_2 \in W(\Gamma_2, \AA)$ gives
rise to the weighting $w=w_1 \circ w_2 \in W(\Gamma, \AA)$, where
$w(e)=w_1(e)$ if $e \in E\Gamma_1$ and $w(e)=w_2(e)$ if $e\in
E\Gamma_2$.

Under the correspondence map, the concatenation of planar networks
corresponds to matrix multiplication:
\begin{equation}
  \label{concat}
M(\Gamma_1 \circ \Gamma_2; w_1 \circ w_2)=M(\Gamma_1; w_1)\cdot M(\Gamma_2; w_2).  
\end{equation}

If $\AA$ is a commutative ring, then we can define
the minors $M_{I,J}(\Gamma; w)$ of the matrix $M(\Gamma; w)$, where $I, J
\subset \{ 1, \dots, n\}$ with $|I|=|J|=k$, but this definition does
not work for a semiring, since it involves signs. The Lindstr\"om Lemma asserts that these
minors can be expressed in terms of multi-paths in $\Gamma$ as follows:
$$
M_{I,J}(\Gamma; w)= \sum_{\alpha \in P_k\Gamma(I,J)} w(\alpha).
$$
Note that the right hand side is well-defined even if $\AA$ is only a semiring.

For $k=1, \dots, n$, we introduce the functions
$$
m_k(\Gamma, w)=\sum_{I, J; |I|=|J|=k} M_{I,J}(\Gamma; w).
$$
If it is clear which planar network is used, we omit $\Gamma$ and use
the shorthand notation $m_k(w)$. When we want to emphasize the
semiring in which $m_k$ takes values, we include it in the notation
$m_k^\AA(w)$.

\begin{ex}
  Let $\T = \mathbb{R} \cup \{ - \infty\}$ be the tropical semiring,
  with addition given by $(x,y) \mapsto \max(x,y)$ and with
  multiplication $(x,y) \mapsto x+y$. The tropical weights are then
  defined by the formula \eqref{tropdef}.
The functions $m^\T_k(\Gamma, w)$ take the form
\begin{equation}\label{tropm}
m^\T_k(\Gamma, w)=\max\{ w(\alpha)| \, \alpha \in P_k\Gamma \} .
\end{equation}
In the case when $P_k\Gamma$ is empty, we set $m^\T_k(\Gamma,
w)=-\infty$ for all weightings $w$.
\end{ex}

Later on, we will see that the functions $m^\T_k$ provide a ``tropical
counterpart'' of the sums of singular values $l^\B_k$. In view of this
analogy, we can introduce ``tropical singular values'' as
$$
\lambda^\T_i(w)=m^\T_i(w) - m^\T_{i-1}(w),
$$
for $i\geq 2$, and $\lambda^\T_1(w) = m^\T_1(w)$. One can show that
$\lambda^\T_i \geq \lambda^\T_{i+1}$ for all $i=1, \dots, n-1$ (the
proof is similar to that of Theorem 2 in \cite{apsz}).

\begin{ex}
  Let $\AA= \T \times U(1)$. The map $(u, \phi) \mapsto \exp(u)\phi$
  from $\AA$ to $\C$ is a homomorphism for the product. We will use the correspondence
  map to define the composition
$$
W(\Gamma, \T \times U(1)) \to W(\Gamma, \C) \to {\rm Mat}_n(\C).
$$
The result is given by the formula
$$
M_{i,j}(\Gamma;u, \phi) = \sum_{\alpha \in P_1\Gamma(i,j)} \exp(u(\alpha)) \phi(\alpha),
$$
where $u$ is a weighting with values in $\T$ and $\phi$ is a weighting with values in $U(1)$.

\end{ex}

\subsection{Results: comparison of different cones}
Let $\Gamma_0$ be the planar network of rank $n$ shown in Figure \ref{net:triang}.

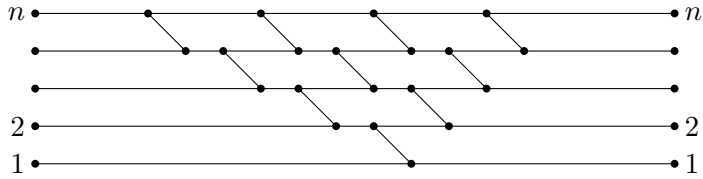
\begin{figure}[h]
\begin{tikzpicture}
\draw (-1,0) -- (7.5,0);
\draw (-1,0.5) -- (7.5,0.5);
\draw (-1,1) -- (7.5,1);
\draw (-1,1.5) -- (7.5,1.5);
\draw (-1,2) -- (7.5,2);
\draw (0.5,2) -- (1,1.5);
\draw (2,2) -- (2.5,1.5);
\draw (3.5,2) -- (4,1.5);
\draw (5,2) -- (5.5,1.5);
\draw (1.5,1.5) -- (2,1);
\draw (3,1.5) -- (3.5,1);
\draw (4.5,1.5) -- (5,1);
\draw (2.5,1) -- (3,0.5);
\draw (4,1) -- (4.5,0.5);
\draw (3.5,0.5) -- (4,0);
\fill[black] (-1,0) circle (1.5pt) node[left]{$1$};
\fill[black] (-1,0.5) circle (1.5pt) node[left]{$2$};
\fill[black] (-1,1) circle (1.5pt);
\fill[black] (-1,1.5) circle (1.5pt);
\fill[black] (-1,2) circle (1.5pt) node[left]{$n$};
\fill[black] (7.5,0) circle (1.5pt) node[right]{$1$};
\fill[black] (7.5,0.5) circle (1.5pt) node[right]{$2$};
\fill[black] (7.5,1) circle (1.5pt);
\fill[black] (7.5,1.5) circle (1.5pt);
\fill[black] (7.5,2) circle (1.5pt) node[right]{$n$};
\fill[black] (0.5,2) circle (1.5pt);
\fill[black] (1,1.5) circle (1.5pt);
\fill[black] (2,2) circle (1.5pt);
\fill[black] (2.5,1.5) circle (1.5pt);
\fill[black] (3.5,2) circle (1.5pt);
\fill[black] (4,1.5) circle (1.5pt);
\fill[black] (5,2) circle (1.5pt);
\fill[black] (5.5,1.5) circle (1.5pt);
\fill[black] (1.5,1.5) circle (1.5pt);
\fill[black] (2,1) circle (1.5pt);
\fill[black] (3,1.5) circle (1.5pt);
\fill[black] (3.5,1) circle (1.5pt);
\fill[black] (4.5,1.5) circle (1.5pt);
\fill[black] (5,1) circle (1.5pt);
\fill[black] (2.5,1) circle (1.5pt);
\fill[black] (3,0.5) circle (1.5pt);
\fill[black] (4,0) circle (1.5pt);
\fill[black] (4,1) circle (1.5pt);
\fill[black] (4.5,0.5) circle (1.5pt);
\fill[black] (3.5,0.5) circle (1.5pt);
\end{tikzpicture}
\caption{The planar network $\Gamma_0$.}\label{net:triang}
\end{figure}

Note that the matrices defined by the correspondence map $M(\Gamma_0;
w)$ are upper-triangular. 

Inspired by the analogy with the multiplicative problem for $\B$
(cf. \eqref{concat}), we can define the following tropical cone:

\begin{equation}\label{tropcone}\begin{split}
\ctropfull =\{(r,s,t)\in\T^{3n};\, &\exists(w_1,w_2)\in W(\Gamma_0, \T)^{\times2}; \\& m^{\T}(w_1) = r,
m^{\T}(w_2)=s, m^{\T}(w_1 \circ w_2) = t \}.
\end{split}\end{equation}
Note that the set of multi-paths $P_k\Gamma_0$ is nonempty for every $k$, 
and hence, we can consider the
``real'' part of this cone:
\begin{equation}
  \label{defrealconet}
  \ctrop = \ctropfull \cap \R^{3n}.
\end{equation}

In \cite{apsz}, we proved the following theorem, which may be thought
of as the solution of the tropical Horn problem.

\begin{theorem}\label{t=kt}
$\ctrop = \conekt$.
\end{theorem}

The main result of this paper is as follows:
\begin{theorem}
$\ctrop = \coneb$.
\end{theorem}

In combination with Klyachko's theorem (Theorem \ref{klyachko}), this result implies $\ctrop=\coneho$.
Together with Theorem \ref{t=kt}, this gives a new proof
of the Knutson--Tao theorem (Theorem \ref{knutao}).

\subsection{The structure of the paper.}

Our purpose in Section \ref{sec:tropical} is to study the relation
between $\ctrop$ and $\coneb$ via the correspondence map
\eqref{correspondence}.  First, in Proposition \ref{limit}, we show
that away from a small set of tropical weights, tropical singular
values approximate the corresponding ordinary singular values
exponentially well.

A refinement of this statement is Proposition \ref{onemprop}, where we
show that this approximation is valid on a large part of $\B$, where
we measure size in terms of the image with respect to the
Gelfand--Zeitlin map.

The main result of the section is Proposition \ref{proph}.
It states that the singular values of the matrices
$A,\,C$, and $AC$, for $(A,C)\in\B\times\B$, are exponentially close 
the corresponding tropical singular values except for a small 
part of $\B \times \B$. This is sufficient to prove the inclusion
$\ctrop\subset\coneb$.

There is a canonical Poisson structure on the space of Hermitian matrices
$\H$, whose symplectic leaves $\H_r$ are
Hermitian matrices with fixed eigenvalues $r\in\R^n$, and
it induces a Liouville measure $\mu_r$ on $\H_r$. Similarly, there is a canonical
Poisson structure on the group $\B$ with symplectic leaves
the upper-triangular matrices with fixed singular values
$\exp(r)$. The corresponding Liouville measure is denoted $\mu_r^\B$.

In Section \ref{poissondh}, we first recall the fact that these measures are
compatible with the corresponding Gelfand--Zeitlin maps (Theorem \ref{integrable} and
Section \ref{gzonb}) in the sense that the pushforwards of the measures of $\mu_r$
and $\mu_r^\B$ onto the Gelfand--Zeiltin polytope are equal to the
Lebesgue measure.  This is a corollary of the complete integrability
of the Gelfand--Zeitlin system.

According to Klyachko's theorem, the images of the map $(K,L)\to
l(K+L)$ on $\H_r\times\H_s$, and the map $(A,C)\to l_\B(AC)$ on $\B_r
\times \B_s$ coincide. This image is a polytope, that we denote by
$\Pi_{r,s}\subset\R^n$. Theorem \ref{measure} is a refinement of this theorem: it
states that the pushforward measures on $\Pi_{r,s}\subset\R^n$ of the
measures $\mu_r \times \mu_s$ and $\mu_r^\B \times \mu_s^\B$ along
these maps coincide; we denote this measure by $\mu_{r,s}$. The proof
of this theorem (given in Appendix) uses the theory of Poisson--Lie
groups.

Combining these two pieces of information about pushforward
measures with our tropical analysis from Section \ref{sec:tropical}, we present our
final argument in the proof of Theorem \ref{measurezero}. Here we consider the
hypothetical exceptional part of $\Pi_{r,s}$ which does not lie in
$\ctrop$, and prove that the measure of this part is zero.

Using  standard arguments from symplectic geometry, this quickly leads
to the proof of our main theorem (Theorem \ref{empty}). We conclude the paper with an interesting
corollary: we provide a tropical description of $\mu_{r,s}$ in Theorem~\ref{tropmeas}.

\vskip 0.2cm

{\bf Acknowledgements.} We are indebted to A. Knutson for posing us a
question of giving a symplectic proof of Horn inequalities. We are
grateful to S. Fomin, E. Meinrenken and C. Woodward for inspiring
discussions. 

The work of A.A. was supported in part by the grant 152812 of the
Swiss National Science Foundation and by the grant MODFLAT of the
European Research Council (ERC).  The works of A.A. and A.S. was
supported in part by the grant 141329 of the Swiss National Science
Foundation.  The work of M.P. was supported in part by the Chebyshev
Laboratory under RF Government grant 11.G34.31.0026, JSC ``Gazprom
Neft'', and RFBR grant 13-01-00935.  Our work was supported in part by
the NCCR SwissMAP of the Swiss National Science Foundation.

\section*{Index of notations}

\bgroup
\def\arraystretch{1.5}
\begin{tabular}{|l| p{0.6\textwidth}|l|}
  \hline
  $\H$ & the set of Hermitian matrices &  \\
  \hline
  $\B$ & the set of upper-triangular matrices with positive diagonal entries & \\ 
  \hline
  $\Cone_{\mathrm{Horn}}$ & Horn cone & Equation \eqref{horncone} \\
  \hline
  $\Cone_{\mathrm{KT}}$ & Knutson--Tao cone defined in terms of hives& Inequalities \eqref{hiveineq}\\ 
  \hline
  $\coneb$ & the multiplicative Horn cone& Equation \eqref{coneb}\\
  \hline
  $\ctropfull$ & the tropical Horn cone & Equation \eqref{tropcone}\\
  \hline
  $P_k\Gamma(I,J)$ & the set of $k$-paths in $\Gamma$ connecting
  sources with labels from $I$ and sinks with labels from $J$  & Section \ref{planar} \\
  \hline
  $W(\Gamma,Q)$ & weightings of network $\Gamma$ with values in the semigroup $Q$ & Section \ref{planar} \\
  \hline
  $\overline W(\Gamma,Q)$ & subset of $W(\Gamma, Q)$ with vanishing
  weightings on all horizontal edges except those adjacent to sinks &
  Section \ref{tropgz} \\ 
  \hline
  $\Phi(\Gamma)$ & weightings of network $\Gamma$ with values in $U(1)$ & Example \ref{phases} \\
  \hline
  $\overline\Phi(\Gamma)$ & weightings from  $\Phi(\Gamma)$ with
  all horizontal edges having weight equal to $1$& Proposition
  \ref{onemprop} \\ 
  \hline
  $l(K)$ &  vector of sums of eigenvalues of $K$ & Section \ref{hornintro} \\
  \hline
  $l^{\B}(A)$ &  multiplicative analog of $l$ & Equation \eqref{notationl} \\
  \hline
  $M(\Gamma;w)$ &  correspondence map & Equation \eqref{correspondence} \\
  \hline
  $m_k^{\T}(\Gamma,w)$ &  maximal value of the weighting $w$ on  $k$-paths  in
  $\Gamma$ & Equation \eqref{tropm} \\
  \hline
  $L_{\mathcal H}$ & the Gelfand-Zeitlin map $\H \to
  \mathbb{R}^\nabla$ recording the values of $l$ on principal submatrices & Section \ref{gzsie} \\
  \hline
  $L_{\B}$ & the map  $\B \to \mathbb{R}^\nabla$ with components given by values of $l^{\B}$ on principal submatrices & Section \ref{preimgz} \\
  \hline
  $L_{\T}$ & the map $W(\Gamma, \T)\to \T^{\nabla}$ with components defined by values of $m^{\T}\Gamma$ on subnetworks & Section \ref{tropgz} \\
  \hline
  $W_\delta(\Gamma,\mathbb R)$ & subset of points of $W(\Gamma,
  \mathbb R)$ with distance $>\delta$ from certain critical hyperplanes & Section \ref{tropestimate} \\
  \hline
  $\Delta_{\mathrm {GZ}}$ & Gelfand--Zeitlin cone defined by interlacing inequalities & Inequalities \eqref{eq:GZ} \\
  \hline
  $\Delta_{0}$ & cone isomorphic to $\Delta_{\mathrm{GZ}}$ via $L_{\T}$ & Proposition \ref{ourcone} \\
  \hline
\end{tabular}
\egroup

\section{The tropical analysis}  \label{sec:tropical}

In this section, we establish tropical approximation estimates for singular values of matrices defined by planar networks. These estimates imply the inclusion of cones $\ctrop \subset \coneb$.

\subsection{Preliminaries}
We begin by recalling some standard facts about interlacing
inequalities and Gelfand--Zeiltin completely integrable systems.

\subsubsection{The Gelfand--Zeitlin system and interlacing
  inequalities}
\label{gzsie}
For a given~$n$, let $S^\nabla$ be the set of maps from the vertices
of the triangular tableau of size $n$ to the set $S$. For instance,
$\R^\nabla$ is the set of triangular tableaux of size $n$ filled with
real numbers $l^k_i$ for $0 \leq i \leq k \leq n$.

We define the {\em Gelfand--Zeitlin} map $L_\H:\H\to\R^\nabla$ as follows. For $A\in\H$, we
assign $l_i(A^{(k)})$ to the $i$th node of the $k$th row for $i>0$, where
$A^{(k)}$ is the principal $k$-by-$k$ submatrix of $A$; we also set
the first element in each row to zero.

The basic result is that, for any $A\in\H$, the tableau $L_\H(A)$ lies
in the {\em Gelfand--Zeitlin} cone $\conegz\subset\R^\nabla$ defined by
the system
\begin{equation}  \label{eq:GZ}
\begin{array}{lll}
l^k_0 =0  & \text{for all }k,&\\
l^{k+1}_{i} + l^{k}_{i-1} & \geq & l^{k+1}_{i-1} + l^{k}_{i},  \\
l^{k+1}_{i} + l^{k}_{i} & \geq & l^{k+1}_{i+1} + l^{k}_{i-1}.
\end{array}
\end{equation}

These inequlities are also called the  {\em interlacing inequalities},
since the numbers $\lambda^{k}_i = l^{k}_i-l^{k}_{i-1}$ (corresponding
to the eigenvalues of Hermitian matrices and their principal
submatrices) satisfy the inequalities
$$
\lambda^{k}_i \geq \lambda^{k-1}_i \geq \lambda^k_{i+1}.
$$

\begin{rem}
  Note that, somewhat surprisingly, the inequalities \eqref{eq:GZ} are
  part of the Knutson--Tao inequalities \eqref{hiveineq}.
\end{rem}

\begin{ex}
For the case of $n=2$, the interlacing inequalities read as follows: 
$$l^2_1 \geq l^1_1\quad\text{ and }\quad l^2_1 +l^1_1 \geq l^2_2.$$
\end{ex}

\subsubsection{Tropical Gelfand--Zeitlin map}\label{tropgz}
Let $\Gamma$ be a planar network of rank~$n$ and let $\AA=\T$ the
tropical semiring.  We denote by $\Gamma^{(k)}$ the maximal subgraph
of $\Gamma$ that does not contain the sinks or sources with
$y$-coordinates above the line $\{y=k\}$. A weighting $w \in W(\Gamma,
\T)$ induces weightings on~$\Gamma^{(k)}$ for all $k$, which, by abuse
of notations, we also denote by $w$.  For each $k=1, \dots, n$,
consider the collection of functions $m^k_i(w)=m^\T_i(\Gamma^{(k)},w),$ for
$i=1, \dots, k$, and place the corresponding values in the $k$-th row
of the triangular tableau (see Figure \ref{hive});
We will call the resulting map $ L_{\T}: W(\Gamma, \T)\to \T^{\nabla}$
the {\em tropical Gelfand--Zeitlin map}.  

\begin{theorem}[\cite{apsz} Theorem 2] \label{mainold} For any planar
  network $\Gamma$ of rank $n$ and any weighting $w \in W(\Gamma,
  \T)$, the components $m^\T_i(\Gamma^{(k)}, w)$ satisfy the
  interlacing inequalities \eqref{eq:GZ}.
\end{theorem}

Let $\Gamma_0$ be the planar network shown in Figure \ref{net:triang}.
Consider the subset $\overline{W}(\Gamma_0, \T) \subset W(\Gamma_0,
\T)$ which consists of weightings $w$ vanishing on all horizontal
edges with the exception of those which end on a sink. Note that the
number of edges carrying non-vanishing weights is then exactly
$N=n(n+1)/2$, which coincides with the number of entries in the
triangular tableaux and with the number of functions $m^k_i$. The
following result is proved in \cite{apsz}:
\begin{proposition} \label{ourcone} There exists a cone $\conez\subset
  \overline{W}(\Gamma_0, \T)$ such that the restriction of $L_\T$ to
  $\conez$ is an isomorphism to $\conegz$.
\end{proposition}

 Proposition \ref{ourcone}, with some abuse of notation, allows
    us to define the bijective inverse
    map, $L_\T^{-1}:\conegz\to\conez$.

\subsection{Tropical estimates for a single matrix}\label{tropestimate}

Let $\Gamma$ be a planar network of rank~$n$.  It will be convenient to work
with the subset
$W(\Gamma,\RR)\subset W(\Gamma,\T)$ of real weightings of $\Gamma$,
considering the reals $\RR=\T\setminus\{-\infty\}$ as a subset of the tropical numbers.
Naturally, the ``multiplication'' in this situation is the addition of
real numbers, thus, for example, for $w\in
W(\Gamma,\RR)$, we have
$$
w(\alpha) = \sum_{e \in \alpha} w(e).
$$

For $\delta>0$, denote by $W_\delta(\Gamma, \R)$ the subset
of weightings $w\in W(\Gamma, \R)$ such that
\begin{itemize}
\item  for any two distinct subsets $\alpha, \beta \subset E\Gamma$,
  we have 
  \begin{equation}
    \label{wcond1}
  |w(\alpha) - w(\beta)|>\delta.
  \end{equation}
  
\item  in all interlacing inequalities \eqref{eq:GZ}  for $m^k_i(w)$, the left hand side is greater than the right
hand side by at least $\delta$  (cf. Theorem \ref{mainold}):
\begin{equation}  \label{wcond2}
\begin{array}{lll}
l^k_0 =0 & \text{for all }k,&\\
l^{k+1}_{i} + l^{k}_{i-1} & > & l^{k+1}_{i-1} + l^{k}_{i} +\delta,  \\
l^{k+1}_{i} + l^{k}_{i} & > & l^{k+1}_{i+1} + l^{k}_{i-1} +\delta.
\end{array}
\end{equation}
 \end{itemize}

 Note that the second condition implies, in particular, that we have
 the gap inequality
$ \lambda_i^\T(\Gamma^{(k)}) - \lambda^\T_{i+1}(\Gamma^{(k)}) >
\delta$ for the tropical eigenvalues of the principal subnetworks of $\Gamma$.

 The complement to the set $W_\delta(\Gamma, \R)$ is contained in the
 $\delta$-neighborhood of a finite number of hyperplanes defined by
 the equations $w(\alpha)=w(\beta)$ and by the equations resulting
 from the interlacing inequalities. 

Recall the definition of the correspondence map \eqref{correspondence}.
The tropical approximation estimate is described by the following proposition:
\begin{proposition}\label{limit}
  Let $\Gamma$ be a planar network of rank $n$, fix
  $\delta>0$, and let $w \in W_\delta(\Gamma, \R)$.
   Then there is a constant $c$ depending only on $\Gamma$,
 such that for $\tau \geq 1$ and for any $\phi\in \Phi(\Gamma)=W(\Gamma, U(1))$, the inequalities
  \begin{equation}
    \label{estlambda}
\left|\frac1\tau\,
   l^\B_i(M(\Gamma;\tau w,\phi))-m^\T_i(\Gamma, w)\right|<
c \cdot e^{-\tau\delta},\quad  i=1,\dots,n,
  \end{equation}
hold.
\end{proposition}
\begin{proof}
Let $\ls_i(A)$ be the elementary symmetric functions of the
singular values of $A$:
\[      1+\sum_{i=1}^n \ls_i(A)\,q^i = \prod_{i=1}^n (1+q \lambda_i(AA^*)) .
\]
 
The determinantal expansion for $AA^*$ with $A=M(\Gamma;\tau w, \phi)$ gives the formula

$$
  \ls_i(M(\Gamma;\tau w,\phi))=
  \sum_{|I|,|J|=i}\left|\sum_{\alpha\in P_i\Gamma(I,J)} 
   \phi(\alpha)\,\exp(\tau w(\alpha))\right|^2. 
$$

Isolating the dominant term $\exp(2\tau m^\T_i(\Gamma, w))$ of the
sum, and using condition \eqref{wcond1} of $w \in W_\delta(\Gamma, \R)$, we obtain
the estimate
\[   \left|\ls_i(M(\Gamma;\tau w,\phi))/
\exp(2\tau m^\T_i(\Gamma, w))\right|\leq 1+c_1 e^{-\tau\delta}
\]
for some constant $c_1$. This implies
\begin{equation}
  \label{ineq1}
   |\log\ls_i(M(\Gamma;\tau w,\phi))-2\tau m^\T_i(\Gamma, w)|\leq  c_1 e^{-\tau\delta}.
\end{equation}

A simple calculation using the second condition \eqref{wcond2}  shows that
$\ls_i(M(\Gamma;\tau w,\phi))$ may be replaced by the dominant term given by the
product of the top $k$ singular values:
\begin{equation}
  \label{ineq2}
  |\log\ls_i(M(\Gamma;\tau w,\phi))-\log\prod_{j=1}^i \, \lambda_j(M(\Gamma;\tau
    w,\phi)M(\Gamma;\tau w,\phi)^*)|<c_2 e^{-\tau\delta}.
\end{equation}
Now, combining \eqref{ineq1} and \eqref{ineq2}, and using  notation
\eqref{notationl}, we obtain the desired inequality \eqref{estlambda} for $\tau \geq 1$.
\end{proof}

\begin{ex}
  Consider the case of $\Gamma_0$ and $n=2$. To simplify things, we
  choose $\phi(e)=1$ for all the edges of the network.

\begin{figure}[h]
\begin{tikzpicture}
\draw (2,0) -- (5,0);
\draw (2,1) -- (5,1);
\draw (3,1) -- (4,0);
\draw(2.8,0) node[above]{$z$};
\draw(4,1) node[above]{$x$};
\draw(3.8,0.2) node[above]{$y$};
\fill[black] (2,0) circle (1.5pt) node[left]{$1$};
\fill[black] (2,1) circle (1.5pt) node[left]{$2$};
\fill[black] (5,0) circle (1.5pt) node[right]{$1$};
\fill[black] (5,1) circle (1.5pt) node[right]{$2$};
\end{tikzpicture}
\caption{The planar network $\Gamma$ for $n=2$.}\label{net:2}
\end{figure}
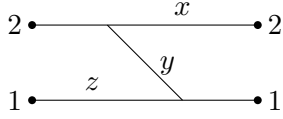
Then we have
$$
m^\T_1(w)=\max(x,y,z), \hskip 0.3cm m^\T_2(w)=x+z.
$$
The correspondence map gives the matrix
$$
M(\Gamma;\tau w)=
\left(
\begin{array}{ll}
e^{\tau x} & e^{\tau y} \\
0 & e^{\tau z}
\end{array}
\right).
$$
For its singular values, we have
$$
l^\B_1(M(\Gamma;\tau w)) =\frac{1}{2} \log\left(\frac{1}{2}\,( U + \sqrt{U^2 - 4V}) \right),  \hskip 0.3cm l^\B_2(M(\Gamma;\tau w)) = \tau(x+y),
$$
where $U=e^{2\tau x} + e^{2\tau y} + e^{2\tau z}$ and $V=e^{2\tau (x+z)}$.
Clearly, $l^\B_2(M(\Gamma;\tau w)) = \tau m^\T_2(w)$, and it is easy to verify that 
$$
\lim_{\tau \to + \infty} \frac{1}{\tau} \, l^\B_1(M(\Gamma;\tau w)) = \max(x,y,z)=m^\T_1(w).
$$

\end{ex}

Proposition \ref{limit} has the following corollary:
\begin{corollary} \label{tropinr} Let $\Gamma$ be a planar network of
  rank $n$. Then the convex polyhedral cone $L^\T(W(\Gamma, \R))$ is
  contained in the smallest closed cone containing the image of $W(\Gamma, \R)$
  under the map
  $L^\B \circ M$.
\end{corollary}

\begin{proof}
Indeed, \eqref{estlambda} implies that the distance of any element of
the interior of the cone
$L^\T(W_\delta(\Gamma, \R))$ from the image of $W(\Gamma, \R)$
  under the map
  $\frac{1}{\tau} L^\B \circ M$
 is
exponentially small as $\tau\to\infty$. But any point of $L^\T(W(\Gamma, \R))$ 
may be approximated by a 
point in $L^\T(W_\delta(\Gamma, \R))$ for $\delta$ small enough.
\end{proof}

\subsection{Preimages under the Gelfand--Zeiltin map}\label{preimgz}
Following Flaschka--Ratiu \cite{FRpreprint}, we introduce  a notion of the 
Gelfand--Zeitlin map for upper triangular matrices. By definition, the
components of the map $L_\B: \B \to \R^\nabla$ are the values
$l_i^\B(A^{(k)})$, where $A^{(k)}$ is the principal $k$-by-$k$ submatrix of $A$.

The following statement is a standard fact from linear algebra:

\begin{proposition}\label{gzprop}
  The image of the map $L_\B$ is the cone $\conegz$, and the fibers of
  $L_\B:\B\to\R^\nabla$ are topological tori of various dimensions (or
  empty). In particular, for an interior point $\Xi\in\conegz$ in the
  Gelfand--Zeitlin cone, the fiber of $L_\B$ is a torus is of the
  maximal possible dimension $n(n-1)/2$.
\end{proposition}

For the convenience of the reader, we sketch a proof of this proposition.
Let $A$ be in $\B_{k+1}$, and set $a=AA^*$ to be the corresponding positive
definite Hermitian matrix.  Denote the eigenvalues of $a$ by $\lambda_0 \geq
\lambda_1 \geq \dots \geq \lambda_k$.  By conjugating the matrix $a$
with an element of $h\in U(k) \subset U(k+1)$, we can bring it to the
form
\begin{equation} \label{a=block}
h a h^{-1}=
\left(
\begin{array}{lllll}
a_{0,0} & a_{0,1}  & a_{0,2} & \dots & a_{0,k} \\
a_{1,0} & \mu_1 &   0           & \dots & 0          \\
a_{2,0} & 0         &  \mu_2   &  \dots & 0       \\
\dots    &  \dots   & \dots      & \dots  & \dots   \\
a_{k,0} & 0          &  0           & \dots & \mu_k
\end{array}
\right).
\end{equation}
Then the condition that $\lambda_i$'s are zeros of the characteristic
polynomial of $a$ gives rise to a system for linear equations on
$a_{0,0}, |a_{0,1}|^2, \dots, |a_{0,k}|^2$, which admits a unique
solution. Hence, the set of matrices $a$ of the form \eqref{a=block}
with given eigenvalues is a torus of dimension at most $k$.  When the
eigenvalues $\lambda_i$ and $\mu_j$ are all distinct, we have
$|a_{0,i}| \neq 0$ for all $i$. In this case, the torus is parametrized by the
angles $\phi_i = {\rm Arg}(a_{0,i})$, and hence it is of dimension
exactly equal to $k$.

Applying this procedure to the natural chain of projections $\B_n \to
\B_{n-1} \to \dots \B_1$, we see that that the fibers of $L_\B$ are
tori of dimension at most $1 + \dots + (n-1) = n(n-1)/2$. When all
eigenvalues of the principal submatrices of $a$ are distinct (i.e.
$L_\B(A)$ is in the interior of the Gelfand--Zeiltin cone), then the
dimension of the fiber is exactly $n(n-1)/2$. \qed

The following proposition describes the preimages of the map $L_\B$
for sufficiently large values of the scaling parameter $\tau$. Introduce the notation
$\conez(\delta)=W_\delta(\Gamma_0, \R) \cap\conez$, $\conegz(\delta) =
L_\T(\conez(\delta))$, and  let $\overline{\Phi} \subset W(\Gamma_0,
U(1))$ stand for the weightings of $\Gamma_0$ 
with values in $U(1)$ taking value $1$ on all horizontal edges.

\vskip 0.2cm

\begin{proposition} \label{onemprop} For every $\delta>0$, there
  exists $\tau_0>0$ such that for all $\tau \geq \tau_0$ the following
  statement holds: for $\Xi\in\conegz(\delta)$ and $A \in
  L_\B^{-1}(\tau\Xi)$, there exist $w\in\conez(\delta/2)$ and $\phi
  \in \overline{\Phi}$ such that $A=M(\Gamma;\tau w,\phi)$.
\end{proposition}

For $ \xi \in \Delta_0$ such that $L_\T(\xi)=\Xi \in \Delta_{{\rm GZ}}(\delta)$, the
point $\tau\Xi$ has the unique preimage $\tau \xi \in \Delta_0$ under
the map $L_\T$.  Proposition \ref{onemprop} states that for
$U_\epsilon$ an $\epsilon$-neighborhood of $\tau \xi$, the image of
$U_\varepsilon \times \overline{\Phi}$ under the correspondence map
$M$ contains the full preimage $L_\B^{-1}(\tau \Xi) \subset \B$.
\[   \begin{tikzcd}
  \Delta_0\times \overline{\Phi} \dar{ }\rar{M} & \B  \dar{L_\B} \\
\Delta_0\rar{L_{\T}} & \Delta_{{\rm GZ}}
\end{tikzcd} \]

\begin{proof}
Fix $\phi\in \overline{\Phi}$ , and consider the map 
\[    f_\phi: \conez\to\conez, \quad
f_\phi(w) = L_\T^{-1}\circ L_\B \circ  M(\Gamma;w,\phi).
\]
Fix $\varepsilon>0$. According to Proposition \ref{limit}, for some $\tau_0$ and $c$, as long as
$w\in W_{\delta/2}(\Gamma_0, \R)$ and $\tau>\tau_0$ we have
$$
\left| \frac{1}{\tau} L_\B \circ M(\Gamma;\tau w, \phi) - L_\T(w) \right| \leq c_1 e^{-\tau \delta/2}.
$$
This inequality can be rewritten as
\begin{equation}
  \label{estimf}
  \left|\frac1\tau L_\T(f_\phi(\tau w))- L_\T(w)\right| \leq c_1\cdot e^{-\tau\delta/2}.  
\end{equation}
Since the restriction of $L_\T$ to $\Delta_0$ is a non-degenerate linear map, this implies
$$
  \left|\frac1\tau f_\phi(\tau w)- w \right| \leq c_2 \cdot e^{-\tau\delta/2}.
$$

Recall the following standard homological argument: let $S$ be a 
convex subset of $\R^{N}$, and  $g: S \to \R^N$ be a continuous map
satisfying $|g(w)-w|<\epsilon$; then the image $g(S)$ contains all
points of $S$ that are at least at the distance $\epsilon$ from its
boundary:
\[     g(S)\supset\{s\in S;\; d(s,\partial S)>\epsilon\}.
\]

Now let $S=\Delta_0(\delta/2)$, $g: w \mapsto \frac{1}{\tau} f_\phi(\tau w)$ and $\epsilon =\delta/2$.
The argument above shows that for large enough $\tau$ the image of the set
$\conez(\delta/2)$ under the map $w\mapsto \frac1\tau f_\phi(\tau w)$
will contain $\conez(\delta)$.

Hence, for every $\phi$, we constructed a point $A_\phi = M(\Gamma;\tau w, \phi)$ in the preimage $L_\B^{-1}(\tau \Xi)$.
By equation \eqref{estimf}, we have
 $$
 |\Xi - L_\T(w)| \leq c_1 \exp(-\delta \tau/2)
 $$ 
 and 
 $$
 |\xi - w| \leq c_2 \exp(-\delta \tau/2).
 $$
By choosing $\tau$ sufficiently large, we can make sure that $|\xi - w| \leq \varepsilon$. Then
$$
M( (\Delta_0(\delta/2) \backslash U_\varepsilon)  \times \overline{\Phi}) \cap L_\B^{-1}(\tau \Xi) = \emptyset .
$$ 
We would like to show that 
$$
L_\B^{-1}(\tau\Xi) \subset M(U_\varepsilon \times \overline{\Phi}).
$$

By Proposition \ref{gzprop}, the preimage $L_\B^{-1}(\tau\Xi)$ is connected. The argument above shows that 
$$
L_\B^{-1}(\tau\Xi)\cap M(\tau\Delta_0(\delta/2) \times \overline{\Phi})
$$
 is a nonempty connected component of
$L_\B^{-1}(\tau\Xi)$. This means that the entire set
$L_\B^{-1}(\tau\Xi)$ is contained in $M(\tau\Delta_0(\delta/2)\times \overline{\Phi})$.

\end{proof}

\subsection{Tropical analysis of the Horn problem}

Now we can pass to our main focus, the Horn problem. The tropical
analysis is quite similar to that of the previous section.

First, we need to define an approproate analog of the set $W_\delta$.
Let $\WW_\delta \subset W(\Gamma_0 \circ \Gamma_0, \R)$ be defined by the following conditions:
\begin{itemize}

\item
for any two distinct subsets $\alpha, \beta \subset \Gamma_0 \circ \Gamma_0$, we have
$| \alpha(w_1 \circ w_2) -\beta(w_1 \circ w_2) | > \delta$,

\item
$L_\T(w_1), L_\T(w_2), L_\T(w_1 \circ w_2) \in \Delta_{{\rm GZ}}(\delta)$.
\end{itemize}

We will also need the corresponding image set $\Sigma(\delta) = L_\T^{\times2} \WW_\delta\subset \R^\nabla\times \R^\nabla$.

\begin{ex}
  Consider the case of $n=2$. The cone $\conegz \times \conegz$ is
  defined by the inequalities $r_2 - r_1 \leq l \leq r_1$ and $s_2-s_1 \leq
  m \leq s_1$. Among others, we have the following inequalities
  defining $\Sigma(\delta)$:
\begin{equation*}\begin{split}
(r_2 - r_1) + \delta < l,\  l < r_1 - \delta,\ (s_2 -s_1) + \delta <m,\  m< s_1 - \delta,\\
|(r_2+s_1) -(l+m+r_1)| >\delta.
\end{split}
\end{equation*}
Note that the first four inequalities state that the point is inside $\conegz(\delta)\times \conegz(\delta)$
whereas the last inequality involves both copies of the cone $\conegz$ at the same time.

\end{ex}

We also introduce the tropical and the usual Horn maps:
$$
\begin{array}{lll}
H_\T(w_1,w_2)  & =  & (m^\T(\Gamma_0,w_1),m^\T(\Gamma_0,w_2),m^\T(\Gamma_0\circ\Gamma_0,w_1\circ w_2)), \\
H_\B(A_1,A_2)  & =  & (l^\B(A_1),l^\B(A_2),l^\B(A_1A_2)).
\end{array}
$$

With these preparations, we can formulate the tropical estimate for
the Horn problem as follows: 

\begin{proposition} \label{proph}
For every $\delta>0$, there exist $\tau_0 >0$ and a constant $c$ such that for every $\tau \geq \tau_0$ the following statement holds:
for $(\Xi_1,\Xi_2)\in \Sigma(\delta)$ and   $A_1 \in L_\B^{-1} (\tau
\Xi_1),\  A_2 \in L_\B^{-1}(\tau \Xi_2)$, there exist $w_1 \circ w_2 \in \WW_{\delta/2}$, 
 and $\phi_1,\phi_2\in \overline{W}(\Gamma_0, U(1))$ such that $M(\Gamma;\tau
w_1,\phi_1)=A_1$, $M(\Gamma;\tau
w_2,\phi_2)=A_2$, and
\begin{equation}
  \label{hineq}
  \left|\frac1\tau H_\B(A_1,A_2)-H_\T(w_1,w_2)\right|\leq ce^{-\delta\tau}.
\end{equation}

\end{proposition}
\begin{proof}
 Indeed, using Proposition \ref{onemprop}, we can conclude that there
are weights $w_1,w_2 \in \overline{W}(\Gamma_0, \R)$  and
 $\phi_1,\phi_2\in \overline{\Phi}(\Gamma_0)$  such that $M(\Gamma;\tau
w_1,\phi_1)=A_1$ and $M(\Gamma;\tau w_2,\phi_2)=A_2$ with
\[ \left|\frac1\tau L_\B(A_1)- L_\T(w_1)\right|<ce^{-\delta\tau}
, \hskip 0.3cm \left|\frac1\tau L_\B(A_2)- L_\T(w_2)\right|<ce^{-\delta\tau}. \]
Since $(\Xi_1, \Xi_2) \in \Sigma(\delta)$, for $\tau$ sufficiently large we have 
$w_1 \circ w_2 \in \WW_{\delta/2}$.

Then the equality 
\[      M(\Gamma;\tau w_1,\phi_1)\cdot M(\Gamma;\tau w_2,\phi_2) = M(\Gamma;\Gamma;\tau w_1\circ w_2,\phi_1\circ \phi_2).
\]
together with  Proposition \ref{limit} yields
\[\left|\frac1\tau L_\B(A_1\cdot A_2)-L_\T(w_1\circ w_2)\right|<ce^{-\delta\tau},\]
which clearly implies \eqref{hineq}.
\end{proof}
\begin{corollary}\label{hornconecor}
  $\ctrop\subset\coneb$
\end{corollary}
Indeed, using Proposition \ref{proph}, we can prove that the smallest
closed cone containing $\coneb$ contains $\ctrop$ exactly in the same way
as we deduced Corollary \ref{tropinr} from Proposition
\ref{limit}. Now, it follows from Klyachko's theorem that
$\coneb$ is itself a closed cone, and hence we can conclude that
$\ctrop\subset\coneb$.

\section{Poisson Geometry and Duistermaat--Heckman measures}\label{poissondh}

\subsection{The Gelfand--Zeitlin system on Hermitian matrices}
Recall that the set of Hermitian matrices $\mathcal{H}$ can be naturally identified with the dual of the Lie algebra $\uu(n)$  by means of the nondegenerate pairing
$$
\langle a, \xi \rangle = {\rm Im} \, {\rm Tr}(a \xi),
$$
where $a \in \mathcal{H}$ and $\xi \in \uu(n)$ (viewed as a
skew-Hermitian matrix). Since $\mathcal{H} \cong \uu^*(n)$, it carries a
linear Kirillov--Kostant--Souriau (KKS) Poisson bracket
$\pi_\H$. Symplectic leaves under this bracket are formed by matrices
with fixed eigenvalues:
$$
\mathcal{H}_r = \{ a \in \mathcal{H}; \,\, l(a)=r\},\text{ where }r \in \R^n.
$$

\begin{ex}
  Consider the case of $n=2$. The space of Hermitian 2-by-2 matrices
  is isomorphic to $\R^4$,
$$
(x,y,z,t) \mapsto
\left(
\begin{array}{ll}
t+z & x+iy \\
x-iy & t -z 
\end{array}
\right).
$$
Under the KKS bracket, $t$ is a Casimir function (i.e., belongs to the
Poisson center), and brackets of the other variables take the form
$$
\{ x, y\} =z,\ \{ y, z\} =x,\ \{z, x\} =y.
$$
The symplectic leaves are either points (if $x=y=z=0$) or 2-spheres:
$$
\mathcal{H}_{(r_1, r_2)} =
\left\{
\left(
\begin{array}{ll}
t+z & x+iy \\
x-iy & t -z 
\end{array}
\right); \,\, r_1=t + \sqrt{x^2+y^2+z^2}, r_2 =2t
\right\} .
$$

\end{ex}

Recall the defitnition of the Gelfand--Zeitlin map $L_\H: \mathcal{H}
\to \R^\nabla$ from Section \ref{gzsie}. The following theorem is
due to Guillemin and Sternberg \cite{guillemin}:

\begin{theorem}\label{integrable}
  The map $L_\H: (\mathcal{H}, \pi_{{\rm KKS}}) \to (\R^N, 0)$ is a
  Poisson map. Its components $l_i=l^{n}_i$ are Casimir
  functions. For $k <n$, the functions $l^{k}_i$ generate a densely defined action
  of a torus of dimension $n(n-1)/2$.

Over each symplectic leaf $\mathcal{H}_r$ with $r\in\R^n$, the map $L_\H$ defines a
completely integrable system in the sense of Liouville--Arnold, i.e.
$$
\{ l^{k}_i, l^{m}_j\} =0,
$$
and the number of independent functions is equal to ${\rm dim}\, \mathcal{H}_r/2$.
\end{theorem}

For generic $r$, the symplectic form on $\mathcal{H}_r$ is given by
\begin{equation} \label{eq:symplectic}
\omega_r = \sum_{k=1}^{n-1} \sum_{i=1}^k d l^{k}_i \wedge d \phi^{k}_i,
\end{equation}
where $\phi^{(k)}_i$'s are some linear 
combinations (with integer coefficients) of the angles defining the torus action. In these coordinates, the Liouville volume form is expressed as
$$
\LL_r = \prod_{k=1}^{n-1} \prod_{i=1}^k dl^{k}_i \wedge d\phi^{k}_i .
$$
Denote by $P_r=L_\H(\mathcal{H}_r)$ the image of $\mathcal{H}_r$ under
the map $L_\H$. This is a convex polytope defined by the interlacing
inequalities and it carries a natural measure $(L_\H)_* \LL_r$ which is equal
to the Lebesgue measure on $P_r$:
$$
(L_\H)_* \LL_r = \chi_{P_r} \,  \prod_{k=1}^{n-1} \prod_{i=1}^k dl^{k}_i .
$$
Here $\chi_{P_r}$ is the characteristic function of $P_r$. Sometimes it
is more convenient to consider the normalized measure
$$
\mu_r = \frac{1}{{\rm vol}(P_r)} \, (L_\H)_* \, \LL_r.
$$

Let $\tau \in \R_+$, and denote by $R_\tau: \R^n \to \R^n$ the
dilation map $R_\tau: r \to \tau r$. Then we have $P_{\tau r} =
\tau P_r$ and
$$
\mu_{\tau r} = (R_\tau)_* \mu_r.
$$

\begin{ex}
For the case of $n=2$, the components of the map $L_\H$ are as follows
$$
l^{2}_1 = t+ \sqrt{x^2+y^2 +z^2}, \hskip 0.2cm l^{2}_2 = 2t, \hskip 0.2cm l^{1}_1 = t-z.
$$
Restricting to the symplectic leaf $\mathcal{H}_r$, we fix $l^{2}_1=r_1$ and $l^{2}_2 = r_2$. The polytope $P_r$ is the closed interval $l^{1}_1 \in [ r_2-r_1  , r_1]$. The corresponding normalized measure is
$$
\mu_r = \frac{1}{2r_1 - r_2} \, \chi_{[r_2 - r_1, r_1]} dl^{1}_1 .
$$
It is invariant under scaling $r_i \mapsto \tau r_i,\ l^{1}_1 \mapsto \tau l^{1}_1$.

\end{ex}

\subsection{The Duistermaat--Heckman measure for the Horn problem}

For $r, s \in \R^n$, the symplectic manifold $\H_r \times \H_s$
carries the diagonal Hamiltonian action of $U(n)$ with moment map
$\Psi_\H: (a,b) \to a+b$. By Kirwan's Convexity Theorem, the image
$$
\Pi_{r,s} = l \circ \Psi_\H  \, (\H_r \times \H_s)
$$
is a convex polytope (the Horn polytope). We clearly have
$$
\Pi_{r,s}=\{ t \in \mathbb{R}^n; \,\, (r,s,t) \in \mathcal{C}_\H\} .
$$
The push-forward of the Liouville measure
$$
{\rm DH}_{r,s} = (l \circ \Phi)_* \, (\LL_r \times \LL_s)
$$
is the Duistermaat--Heckman measure on $\Pi_{r,s}$. By the
Duistermaat--Heckman Theorem,  $\mathrm{DH}_{r,s}$ is absolutely continuous with respect
to the Lebesgue measure on $\Pi_{r,s}$, and  the corresponding Radon--Nykodim
derivative is a piece-wise polynomial function, which is strictly
positive on the interior of $\Pi_{r,s}$. We can again define the
normalized measure
$$
\mu_{r,s}=\frac{1}{{\rm vol}(P_r) {\rm vol}(P_s)} \, {\rm DH}_{r,s},
$$
with the obvious scaling property:
$$
\Pi_{\tau r, \tau s}=\tau \Pi_{r,s}, \hskip 0.3cm \mu_{\tau r,\tau s}=(R_\tau)_* \mu_{r,s}.
$$

\begin{ex}
Let $n=2$ and choose $r_1=r,\, r_2=0,\, s_1=s$, and $s_2=0$. Then it is easy to check that the Horn polytope is of the form
$$
\Pi_{r,s} = \{ (t,0) \in \R^2; \,\, |r-s| \leq t \leq r+s \}.
$$
It is equipped with the normalized measure
$$
\mu_{r,s} = \frac{1}{2rs}\,  \chi_{[|r-s|, r+s]} \, tdt.
$$
This measure is invariant under scaling $r\mapsto \tau r, s \mapsto \tau s, t\mapsto \tau t$.

\end{ex}

\subsection{The Gelfand--Zeitlin system on the group $\B=U^*(n)$}\label{gzonb}
The group $\B$ carries a natural action of the group $U(n)$.
Recall that the Iwasawa decomposition
\[ g = A u,\;\mathrm{for}\; g \in \mathrm{Gl}(n, \C), A \in \B,\ \mathrm{and}\ u \in U(n),
\]
gives rise to the identification
$$
\B \cong \glnc/\Un
$$
of the group $\B$ with a homogeneous space. This presentation defines
 a natural action of $\glnc$ on $\B$ by
multiplication on the left. The restriction of this action to the
subgroup $\Un$ is called the {\em dressing action}. For $x \in
u(n)$ we denote by $\xi_x$ the corresponding fundamental
vector field on $\B$.

The group $\B$ has a canonical multiplicative Lu--Weinstein
Poisson structure~$\pi_\B$, that is defined as follows. Let $dA\,A^{-1}$ be
the right-invariant Maurer--Cartan 1-form on $\B$ with values in the
Lie algebra ${\rm Lie}(\B)$.  There is a canonical pairing between ${\rm Lie}(\B)$
and $\mathfrak{u}(n)$ given by
\[
\langle \xi, x \rangle = {\rm Im} \, {\rm Tr}(\xi x).
\]
The bivector $\pi_\B$ is the unique bivector on $\B$ such that
\[
\pi_\B( \langle dA\,A^{-1}, x \rangle, \cdot) = \xi_x.
\]
Note that for $x$ a diagonal skew-Hermitian matrix, we have
\[
\langle dA\,A^{-1}, x \rangle = d \, \langle \log(A_d), x\rangle,
\] 
where $A=A_d N(A)$ with $A_d$ a diagonal matrix and $N(A)$ a unipotent
upper-triangular matrix.  Hence, the action of the Cartan subgroup of
$U(n)$ consisting of unitary diagonal matrices is Hamiltonian with the
moment map $\Psi(A)= \log(A_d)$.

\begin{ex}
For $n=2$, we have a parametrization
$$
A =
\left(
\begin{array}{ll}
y & z \\
0 & y^{-1} 
\end{array}
\right),
$$
with $y \in \R_+$ and $z \in \C$. The Poisson brackets read
$$
\{ y, z\}_\B = \frac{i}{2} \, yz, \hskip 0.2cm \{ y, \overline{z}\}_\B
= - \frac{i}{2} \, y\overline{z}, \hskip 0.2cm \{ z,
\overline{z}\}_\B= i(z^2 - \overline{z}^2).
$$
The dressing action of the diagonal circle ${\rm diag}(\exp(i\theta),
\exp(-i\theta))$ is given by $y \mapsto y, z\mapsto
z\exp(2i\theta)$. The moment map for this action is $\Psi(A)=\log(y)$.
\end{ex}

Symplectic leaves of the Poisson structure $\pi_\B$ are orbits of the
dressing action and at the same time fibers of the map $l^\B$. For $r
\in \R^n$, we denote by $\B_r = (l^\B)^{-1}(r)$ the
corresponding symplectic leaf.
The leaf $\B_r$
consists of matrices $A \in \B$ such that the eigenvalues of $AA^*$
are given by $(\exp(r_1), \dots, \exp(r_n))$.

Similarly to the case of Hermitian matrices, the map $L_\B$ defines a
completely integrable system on each leaf $\B_r$. The image
$L_\B(\B_r)$ is the same polytope $P_r$ defined by the interlacing
inequalities. Moreover, in action-angle variables $(l^{k}_i,
\phi^{k}_i)$ the symplectic form is again described by equation
\eqref{eq:symplectic} and the induced normalized measure on $P_r$ is
again the measure $\mu_r$ (see \cite{FRpreprint} for details).

\subsection{Duistermaat--Heckman measure for the multiplicative problem}
For a pair of symplectic leaves $\B_r$ and $\B_s$, the space $\B_r \times \B_s$ carries the dressing action of $U(n)$ defined by $g: (A,B) \mapsto (A', B')$, where
$$
gA = A' g',\ g'B=B'g''
$$
with $g,g',$ and $g''$ in $U(n)$ and $A,A',B,B' \in \B$. This action has a
moment map in the sense of Lu, $\Psi(A, B)=AB$.  By the Klyachko
Theorem, the composition map $l^\B \circ \Psi$ sends $\B_r \times
\B_s$ to the same polytope $\Pi_{r,s} \subset \R^n$ as in the case of
Hermitian matrices. Denote the normalized push-forward measure on
$\Pi_{r,s}$ by $\mu^\B_{r,s}$. 
\begin{theorem}  \label{measure}
$$
\mu^\B_{r,s} = \mu_{r,s}.
$$
\end{theorem}

We give a proof of this theorem in  Appendix.

An important corollary of this Theorem is the scaling invariance of $\mu^\B_{r,s}$:
$$
\mu^\B_{\tau r, \tau s} = (R_\tau)_* \mu^\B_{r,s},
$$
which follows from the analogous (obvious) property $\mu_{r,s}$.

\section{Comparison of the multiplicative and tropical Horn problems}

In this section, we put all the elements of our argument together, and
prove the equivalence of the multiplicative and tropical Horn
problems.

For $r,s \in \R^{n}$, define the polytope
$$
\Pi^\T_{r,s} = \{ t \in \R^{n}; \,\, (r,s,t) \in \ctrop\}.
$$
Since $\ctrop \subset \mathcal{C}_\B=\mathcal{C}_{{\rm Horn}}$, we
have $\Pi^\T_{r,s} \subset \Pi_{r,s}$. We introduce the exceptional set
$$
X_{r,s} = \Pi_{r,s} \backslash \Pi^\T_{r,s}.
$$
Clearly, showing that $\ctrop=\coneb$ is equivalent to showing that
$X_{r,s}$ is empty.

\begin{theorem}\label{measurezero}
$$
\mu_{r,s}(X_{r,s})=0.
$$
\end{theorem}

\begin{proof}
Choose $\delta>0$, $\varepsilon >0$ and $\tau>0 $ such that $\tau> \tau_0$ with $\tau_0$ defined in  Proposition \ref{proph} and
$c \exp(-\tau \delta) < \varepsilon$, where $c$ is the maximum of the constants corresponding to the graphs $\Gamma_0$ and $\Gamma_0 \circ \Gamma_0$.

Let
$$
(u, v) \in (P_{r} \times P_{s}) \backslash \Sigma(\delta)
$$
and consider a pair $(A, B) \in \B_{\tau r} \times \B_{\tau s}$ such that $L_\B(A)=\tau u, L_\B(B)=\tau v$.

Then, by Proposition \ref{proph}, there exist weights $w_1, w_2 \in  W(\Gamma_0, \R)$ and $\phi_1, \phi_2 \in \overline{\Phi}(\Gamma_0)\subset \overline{W}(\Gamma_0, U(1))$ such that  $A=M(\Gamma;\tau w_1, \phi_1)$ and $B=M(\Gamma;\tau w_2, \phi_2)$ with
\begin{equation} \label{two}
|L_\T(w_1) -u| < \varepsilon,  \hskip 0.3cm
|L_\T(w_2) -v| < \varepsilon,
\end{equation}
and
\begin{equation} \label{three}
\left| \frac{1}{\tau}\, L_\B(AB) - L_\T(w_1 \circ w_2)\right| < \varepsilon .
\end{equation} 
Denote by $w_u=L_\T^{-1}(u)$ and $ w_v=L_\T^{-1}(v)$ the preimages of $u$ and $v$ under the map $L_\T$.
Since $L_\T$ is a non-degenerate linear map, the inequalities \eqref{two} imply
$|w_u - w_1| \leq c_1 \varepsilon,\ |w_v - w_2| \leq c_1 \varepsilon$ and
$$
|L_\T(w_u \circ w_v) - L_\T(w_1 \circ w_2)| < c_2 \, \varepsilon,
$$
where $c_i$'s are appropriately chosen constants.
Then, combining this with  inequality \eqref{three}, we obtain
$$
\left| \frac{1}{\tau}\, L_\B(AB)- L_\T(w_u \circ w_v)\right| < c_3\,  \varepsilon  ,
$$
and, as a consequence,
$$
\left| \frac{1}{\tau}\, l^\B(AB) - m^\T(w_u \circ w_v)\right| < c_3 \, \varepsilon .
$$

Hence,
$$
l^\B\left( (L_\B \times L_\B)^{-1} (P_{\tau r} \times P_{\tau s}) \backslash \Sigma(\tau \delta) ) \right) 
\subset U_{\tau\varepsilon}(\Pi^\T_{\tau r, \tau s}),
$$
where $U_{\tau\varepsilon}(\Pi^\T_{\tau r, \tau s})$ is the $\tau
\varepsilon$ neighborhood of the polytope $\Pi^\T_{\tau r, \tau
  s}$. This implies that
$$
\mu_{\tau r, \tau s}(X_{\tau r, \tau s}) \leq c_4 \delta + c_5 \varepsilon,
$$
where $c_4$ is the total volume of intersections of the hyperplanes
defining $\Sigma(\delta)$ with $P_{r} \times P_{s}$ and $c_5$ is the
total volume of the boundary of $\Pi^\T_{r, s}$.

 Recall that
$\mu_{\tau r, \tau s}(X_{\tau r, \tau s}) = \mu_{r,s}(X_{r,s})$ which
implies
$$
\mu_{r,s}(X_{r,s}) \leq c_4 \delta + c_5  \varepsilon .
$$
Since the constants $\delta$ and $\varepsilon$ were chosen
arbitrarily, we can conclude that \mbox{$\mu_{r,s}(X_{r,s})=0$}, as required.
\end{proof}

\begin{theorem}\label{empty}
$X_{r,s} = \emptyset$ and $\Pi_{r,s}=\Pi^\T_{r,s}$.
\end{theorem}

\begin{proof}
  The polytope $\Pi_{r,s}$ is the image of the symplectic manifold
  $\H_r \times \H_s$ under the composition of the moment map $(a,b)
  \to a+b$ with the map $l$. By the Duistermaat--Heckman theorem,
  the induced measure $\mu_{r,s}$ on $\Pi_{r,s}$ is piece-wise
  polynomial and non-vanishing on its interior. Since both $\Pi_{r,s}$
  and $\Pi^\T_{r,s}$ are closed polytopes, their difference $X_{r,s}$
  is either empty or contains points of the interior of
  $\Pi_{r,s}$. The vanishing of the measure $\mu_{r,s}(X_{r,s})=0$
  implies that $X_{r,s}$ contains no points in the interior of
  $\Pi_{r,s}$. Then it must be empty, as required.
\end{proof}
This construction has the following interesting corollary:
\begin{theorem}\label{tropmeas}
Let $\kappa: P_r \times P_s \to \Pi_{r,s}$ be the map defined by 
$$
\kappa(u, v) = m^\T(L_\T^{-1}(u) \, \circ \, L_\T^{-1}(v)).
$$
Then
$$
\kappa_*(\mu_r \times \mu_s) = \mu_{r,s}.
$$
\end{theorem}

\begin{proof}
The measure $\mu_{r,s}$ is piece-wise polynomial by the Duistermaat--Heckman Theorem. The measure
$$
\mu^\T_{r,s}=\kappa_*(\mu_r \times \mu_s)
$$
is the image of the Lebesgue measure on $P_r \times P_s$ under a piece-wise linear map. Hence, $\mu^\T_{r,s}$ is also piece-wise polynomial, and to establish the equality $\mu^\T_{r,s}=\mu_{r,s}$, it is enough to compare them on open balls. Let $B \subset \Pi_{r,s}$ be an open ball. Then, similar to the proof of the previous theorem, we have an estimate
$$
|\mu^\T_{r,s}(B) - \mu_{r,s}(B)| < c_4 \delta + c_6 \varepsilon,
$$
where $c_4$ (as before) is the total volume of intersections of the hyperplanes defining $\Sigma$ with $P_r \times P_s$, and $c_6$ is the volume of the boundary of $B$. As the constants $\delta$ and $\varepsilon$ can be chosen arbitrarily, we conclude that $\mu^\T_{r,s}(B) = \mu_{r,s}(B)$. Hence, $\mu^\T_{r,s}=\mu_{r,s}$, as required.  
\end{proof}

\begin{ex}
Let $n=2$ and $r=(r,0), s=(s,0)$. Then, $P_r=[-r, r], P_s=[-s, s]$. The map $\kappa$ is of the form
$$
\kappa: (x,y) \mapsto {\rm max}(r-y, x+s).
$$
It is easy to see that  ${\rm im}(\kappa)=[|r-s|, r+s] = \Pi_{r,s}$. For $t \in [|r-s|, r+s]$, and its preimage is a union of two segments:
$$
\kappa^{-1}(t)=[-r, t-s] \times \{ r-t\} \, \cup \, \{ t-s\} \times [r-t, s].
$$

\begin{figure}[h]
\begin{tikzpicture}
  \draw [->] (-4,0) -- (4,0) node[right] {$x$};
  \draw [->] (0,-2.5) -- (0,2.5) node[above] {$y$};
  \draw (-3,-1.5) -- (3,-1.5);
  \draw (-3,1.5) -- (3,1.5);
  \draw (-3,-1.5) -- (-3,1.5);
  \draw (3,-1.5) -- (3,1.5);
  \draw (-3,0) node [below left] {$-r$};
  \draw (3,0) node [below right] {$r$};
  \draw (0,-1.5) node [below left] {$-s$};
  \draw (0,1.5) node [above left] {$s$};
  \draw [ultra thick] (-3,-1) -- (2.5,-1) -- (2.5,1.5);
\end{tikzpicture}
\caption{$P_r\times P_s$ and $\kappa^{-1}(t)$.}
\end{figure}
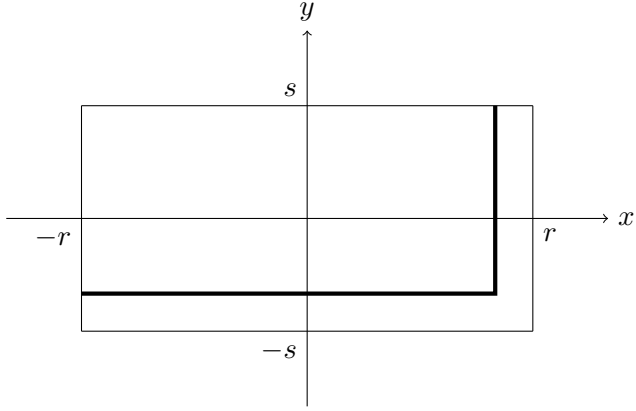

The induced measure on $\Pi_{r,s}$ is given by
\begin{equation*}\begin{split}
&\kappa_*\left(\frac{1}{2r}\chi_{[-r,r]}dx \times \frac{1}{2s}\chi_{[-s,s]}dy\right) = \\ 
&= \frac{1}{4rs}\chi_{[|r-s|,r+s]}\left( (t-s+r) + (s-r+t)\right) dt = \frac{1}{2rs}\chi_{[|r-s|,r+s]}t dt,
\end{split}\end{equation*}
which coincides with the measure $\mu_{r,s}$.
\end{ex}

\section*{Appendix: proof of Theorem \ref{measure}}
\begin{appendix}

In order to prove this theorem, we need to use several facts from symplectic geometry.
First, we recall the Linearization Lemma of \cite{alekseev}.
\begin{theorem}[Linearization Lemma] \label{linearization} Let $G$ be
  a compact connected semisimple Lie group, and $\mathfrak{g}$ its Lie
  algebra. Then there exist a 2-form $\nu\in\Omega^2(\mathfrak g^*)$
  and a map $u:\mathfrak g^*\rightarrow G$ such that for any
  symplectic manifold $(M,\omega)$ endowed with a Poisson $G$-action
  and a map $\Psi:M\rightarrow G^*$, which is a moment map in the sense
  of Lu, the triple $(M,\varpi=\omega-\psi^*\nu,\psi={\rm
    Exp}^{-1}\circ\Psi)$ is a Hamiltonian $G$-space with the moment
  map~$\psi$.

  In addition, in this case, the map $u_M:M\rightarrow M,\
  u_M(m)=u(\psi(m))\cdot m$ is a symplectomorphism between
  $(M,\varpi)$ and $(M,\omega)$.
\end{theorem}

In particular, linearizations of dressing orbits on $G^*$ are
coadjoint orbits in $\mathfrak{g}^*$.  For example, let $G=U(n)$ with
the standard Poisson structure; then $G^*=\mathcal{B}$ with
the Lu--Weinstein Poisson structure. The indentification of $\B$ with
$\H$ via the equation $AA^*=\exp(2K)$ induces an isomorphism of $G$-spaces
$\B_r \cong \H_r$ between dressing and coadjoint orbits. 
Moreover, the linearization of $\B_r$ with the Lu--Weinstein Poisson structure
is exactly $\H_r$ with the KKS Poisson structure.

For a $G$-Hamiltonian space $(M, \varpi, \psi)$, we can define the
projection map onto the positive Weyl chamber
$\sigma= \pi \circ \psi: M \to W_+$ and the corresponding Duistermaat--Heckman
measure $\Dh_0(M, \varpi, \psi)= \sigma_* \, \varpi^N/N! $, where
$N={\rm dim}\, M/2$.  Similarly, we define the measure
$\Dh(M,\omega,\Psi)=\sigma_* \, \omega^N/N!$.

\begin{lemma}  \label{lemma:lin}
In the setup of Theorem \ref{linearization}, we have
$\Dh(M,\omega,\Psi)=\Dh_0(M, \varpi, \psi)$.
\end{lemma}
\begin{proof}
  According to Theorem \ref{linearization}, there is a map
  $u_M:M\to M$ such that
  $u_M^*\omega=\varpi$. This implies
  $(u_M)_*\varpi^N=\omega^N$, and therefore we have
\begin{equation*}
\begin{split}
\Dh(M,\omega,\Psi)&=\sigma_*\left(\frac{\omega^N}{N!}\right)=\sigma_*\circ(u_M)_*\left(\frac{\varpi^N}{N!}\right)=\pi_*\circ\psi_*\circ(u_M)_*\left(\frac{\varpi^N}{N!}\right)=\\
&=\pi_*\circ(Ad_u^*)_*\circ\psi_*\left(\frac{\varpi^N}{N!}\right)=\Dh_0(M,\varpi,\psi).
\end{split}
\end{equation*}
Here we  used that $\pi \circ {\rm Ad}^*_g = \pi$ for any $g \in G$.
\end{proof}

Next, recall  Theorem 3.4 of \cite{amw}.
\begin{theorem}[``Linearization commutes with products''] \label{thm:lincommutes}
\label{lincomm}
  Let $(M_i, \omega_i, \Psi_i)$, $i=1,2$, be two Hamiltonian
  $G$-spaces endowed with moment maps in the sense of Lu, and let
  $(M_i, \varpi_i, \psi_i)$, $i=1,2$, be the corresponding
  linearizations (cf. Theorem \ref{linearization}).  Then the space
  $(M_1 \times M_2, \omega= \omega_1 + \omega_2, \Psi(m_1,
  m_2)=\Psi_1(m_1) \Psi_2(m_2))$ carries a Poisson $G$-action for
  which $\Psi$ is the moment map in the sense of Lu, and the
  linearization of this space is $G$-equivariantly symplectomorphic to
  the $G$-Hamiltonian space $(M_1 \times M_2, \varpi=\varpi_1 +
  \varpi_2, \psi = \psi_1 + \psi_2)$. Moreover, the symplectomorphism
  $\xi: M_1 \times M_2 \to M_1 \times M_2$ intertwines the moment maps 
  ${\rm Exp}^{-1} \circ \Psi$ and $\psi$.
\end{theorem}

Since the $G$-equivariant symplectomorphism $\xi$ intertwines the moment maps,
the Duistermaat--Heckman measures of the corresponding $G$-Hamiltonian spaces coincide, 
and we obtain the following corollary of Lemma \ref{lemma:lin} and Theorem \ref{thm:lincommutes}:
\begin{corollary}\label{dhcor}
  In the setup of Theorem \ref{lincomm}, we have 
$$
\Dh(M_1 \times M_2, \omega, \Psi) = \Dh_0(M_1 \times M_2, \varpi, \psi).
$$
\end{corollary}
We are now ready to prove Theorem \ref{measure}.
\begin{proof}[Proof of Theorem \ref{measure}]
Let $G=U(n), G^*=\mathcal{B}$ and $M_1 = \B_r \cong \H_r, M_2 =\B_s \cong  \H_s$. 
Then, $\Dh_0(M_1 \times M_2, \varpi, \psi)=\mu_{r,s}$ and $\Dh(M_1
\times M_2, \omega, \Psi)=\mu^\B_{r,s}$, and, the equality
$\mu^\B_{r,s}=\mu_{r,s}$ follows from Corollary \ref{dhcor}.
\end{proof}
\end{appendix}

\end{document}